\documentclass{amsart}%%%%%%%%%%%%{article}
\usepackage {amssymb}
\usepackage {amsmath}
\usepackage{amsthm}
\usepackage{graphicx}
\usepackage {amscd}
\usepackage {epic}
\usepackage[alphabetic]{amsrefs}
\usepackage[colorlinks, citecolor=blue]{hyperref}

\newcommand{\sX}{{\mathcal{X}}}

\newcommand{\D}{{\mathcal{D}}}
\newcommand{\sB}{{\mathcal{B}}}
\newcommand{\sA}{{\mathcal{A}}}

\newtheorem{thm}{Theorem}[section]

\newtheorem{lem}[thm]{Lemma}

\newtheorem{cor}[thm]{Corollary}

\newtheorem{prop}[thm]{Proposition}

\newtheorem{conj}[thm]{Conjecture}
\theoremstyle{definition}
\newtheorem{defn}[thm]{Definition}

\newtheorem{notation}[thm]{Notation}   
  
\newtheorem{defn-thm}[thm]{Definition--Theorem}  %!!!!!!!!!!!!!!!!!!!!!!!!
\newtheorem{defn-lem}[thm]{Definition--Lemma}  %!!!!!!!!!!!!!!!!!!!!!!!!
\theoremstyle{remark}
\newtheorem{claim}[thm]{Claim}

\input{diagrams.sty}
\input{xy}
\xyoption{all}

\begin{document}
\title{Boundedness of log Calabi-Yau pairs of Fano type}
\author{Christopher D. Hacon}
\date{\today}
\address{Department of Mathematics \\
University of Utah\\
155 South 1400 East\\
JWB 233\\
Salt Lake City, UT 84112, USA}
\email{hacon@math.utah.edu}
\author{Chenyang Xu}
\address{Beijing International Center of Mathematics Research, 5 Yiheyuan Road, Haidian District, Beijing 100871, China}
\email{cyxu@math.pku.edu.cn}

\begin{abstract}
We prove a boundedness result for klt pairs $(X,B)$ such that $K_X+B\equiv 0$ and $B$ is big. As a consequence we obtain a positive answer to the Effective Iitaka Fibration Conjecture for  klt pairs with big boundary. 
\end{abstract}
\thanks{We would like to thank Lev Borisov and James M$^{\rm c}$Kernan for helpful communications. The first author was partially supported by NSF research grants no: DMS-1300750, DMS-1265285 and a grant from the Simons foundation, the second author was partially supported by the grant `The Recruitment Program of Global Experts'. A large part of this work was done when the second author visited IAS, which was partially sponsored by Ky Fan and Yu-Fen Fan Membership Funds, S.S. Chern Fundation and NSF: DMS-1128155, 1252158. 
} 
\maketitle
\tableofcontents
%%%%%%%%%%%%%%%%%%%%%%%%%%%%%
\section{Introduction}
Let $X$ be a smooth complex projective variety. The canonical ring $R(K_X)=\oplus _{m\geq 0}H^0(\omega _X^m)$ is one of the fundamental birational invariants used to study the geometry of $X$. 
Since $R(K_X)$, is finitely generated (\cite{BCHM10, CL12}), if $R(K_X)\ne \mathbb C$, then there is a natural rational map $X\dasharrow Z={\rm Proj} R(K_X)$. This map is known as the Iitaka fibration and it is equivalent to the map $\phi _m:X\dasharrow |mK_X|$ for any $m>0$ sufficiently divisible.
It is an important problem to understand for what values of $m$, the map $\phi _m$ is birational to the Iitaka fibration.
It is classically known that for curves $m\geq 3$ suffices and for surfaces $m\geq 24$ suffices (in fact by work of 
Bombieri, if $X$ is of general type, then $m\geq 5$ is sufficient). %and ??? For 3 folds of general type, $\phi_m$ is birational onto its image for all $m\geq 77$ (\cite{CC}), if $\kappa (X)=0$ (Morrison). (How about $\kappa =1$ Fujino-Mori? and $\kappa =2$?)
This leads to the following natural conjecture (cf. \cite{HM06}).
\begin{conj}\label{c-weak}  Fix a positive integer $n$.  Then there exists an integer $k$, which only depends on $n$, such that if $X$ is an $n$-dimensional smooth complex projective variety, then  $|kK_X|$ defines a map birational to the Iitaka fibration. 
 \end{conj}

For threefolds, the conjecture is answered affirmatively in \cite{Kawamata86, FM00, VZ00, CC}. We remark that for threefolds, such $m$ can be computed explicitly. However, in many cases  finding the optimal bound is a difficult task. 
Furthermore,  in  \cite{Tsuji06, HM06, Takayama06}, Conjecture \ref{c-weak} is verified for varieties of general type and arbitrary dimension, however no explicit expression for $m$ is known. 
Finally, the recent work \cite{BZ14} gives a positive answer under the additional assumption that the general fiber $F$ of the Iitaka fibration is bounded in an appropriate sense (more precisely, the smallest integer $b>0$ such that $|bK_F|\ne 0$ is bounded as well as the middle Betti number of a desingularization of the corresponding degree $b$ cyclic cover $\tilde F \to F$).

In many applications, it is important to study log pairs. However,  boundedness result for log pairs do not hold if we allow arbitrary coefficients. In fact, as observed in \cite{Kollar94}, the right condition for the coefficients of the boundary should be that they satisfy the descending chain condition (DCC). With this assumption, we can generalize the above conjecture to the context of log pairs.  

\begin{conj}[Effective Log Iitaka Fibration Conjecture] \label{c-strong} 
Fix a positive integer $n\in \mathbb{N}$ and a DCC set $I\subset [0,1]\cap \mathbb{Q}$. Then there exists an integer $m>0$ which only depends on $n$ and $I$ such that if $(X,B)$ is  an $n$-dimensional complex projective  log canonical pair such that the coefficients of $B$ are in $I$ and $\kappa (K_X+B)\geq 0$, then $|m(K_X+B)|$ defines a map birational to the Iitaka fibration. 
 \end{conj}

Various low dimensional cases of Conjecture \ref{c-strong} are known (see \cite{Kollar94, Alexeev94, Jiang13, Todorov10, TX09}). In arbitrary dimension, we also have an affirmative answer in the log general type case (cf. \cite{HMX13, HMX14}).  
However, the case of intermediate Kodaira dimension is more complicated and very little is known in arbitrary dimension.
 The well-known strategy, first developed in \cite{FM00}, is to use the canonical bundle formula, which is a far reaching generalization of Kodaira's formula for elliptic fibrations, to reduce the above question to a question about the general fiber and the base pair which is of log general type.  
 % coefficients are contained in a set $I$ satisfying the DCC condition.

More precisely, let $f:X\to Z$ be the Iitaka fibration, then the pluricanonical ring $R(K_X+B)$ can be described in terms of the pluricanonical type ring $R(K_Z+B_Z+M_Z)$ where $(Z,B_Z)$ is a klt pair, the coefficients of $B_Z$ belong to an appropriate DCC set and $K_Z+B_Z+M_Z$ is big. Here $B_Z$ denotes the boundary part which measures the singularities of the map $f:(X,B)\to Z$ and $M_Z$ is the moduli part which measures the variation of the fibers of $f:(X,B)\to Z$.
If one can show that $(Z,B_Z+M_Z)$ is also a log canonical pair and the coefficients of $M_Z$ also belong to a DCC  (or finite) set, then the result would follow from the results in \cite{HMX13, HMX14}.
Therefore, the main difficulty of this approach is to control the corresponding moduli divisor $M_Z$.
This is a difficult question. An affirmative answer would follow from the effective adjunction conjectures in \cite[Conjecture 7.13]{PS06}.
In \cite{BZ14}, the authors treat $M_Z$ as a divisor class (instead of an actual divisor) and develop results similar to those of \cite{HMX14} in this context. 

 In this note,  we show that the results of \cite{HMX14}, imply a (strong) boundedness result for klt pairs $(X,B)$ such that $K_X+B\equiv 0$ where $B$ is big.
 
 \begin{thm}\label{t-k=0}
Fix $I\subset [0,1]\cap \mathbb Q$ a DCC set and $n\in \mathbb N$, then there exist $k>0$ which only depends on $I$ and $n$ such that if $(X,B) $ is klt,  $B$ is big, $\dim X =n$, the coefficients of $B$ are contained in $ I$ and $\kappa (K_X+B)=0$, then $h^0(k(K_X+B))=1$.
If moreover $K_X+B\sim _{\mathbb Q }0$, then $(X,B)$ belongs to a bounded family.
\end{thm}
 
 As a consequence, by \cite{BZ14}, we also obtain a positive answer for \eqref{c-strong} when the general fibers of the Iitaka fibration have a big boundary.
\begin{thm}\label{t-bd} Fix  a DCC set $I\subset [0,1]\cap \mathbb{Q}$ and $n\in \mathbb N$ then there is an integer $k>0$ depending only on $n$ and $I$ such that if \begin{enumerate}
\item $(X,B)$ is an $n$-dimensional klt pair,
\item the coefficients of $B$ are in $I$,
\item  $ \kappa (K_X+B)\geq  0$ and $(X,B)\dasharrow Z$ is the Iitaka fibration, and
\item if $\eta$ is the generic point of $Z$ then $B_{\eta}$ is big (see Definition \ref{d-generic}), 
\end{enumerate}
 then $|k(K_X+B)|$ is birational to the Iitaka fibration of $(X,B)$.
\end{thm}

%\begin{rmk} It would be interesting to investigate the above result under the assumption that $(X,B)$ is log canonical (instead of klt). % and to give a "direct proof" avoiding the difficult arguments necessary to control the moduli part $M_Z$ (cf. \cite{BZ14}.)
%\end{rmk} 
 
As an immediate corollary of \eqref{t-k=0}, we obtain the following result which can be used to generalize \cite[Theorem 5.2]{BL14} from birational boundedness to boundedness (see \cite[3.4]{BL14}). 
\begin{cor}\label{c} Fix two positive integers $n$ and $m$. Consider the set of $n$-dimensional weak Fano varieties $X$ with Cartier index $m$, i.e., $X$ is klt, $-K_X$ is big and nef and $-mK_X$ is Cartier. Then the set of  all such $X$ belongs to a bounded family. 
\end{cor}
 \section{Preliminaries}
\begin{notation} We follow the usual notation from \cite{KM98} and \cite{BCHM10}. We say that a pair $(X,B)$ is $\epsilon$-lt ($\epsilon$ log terminal), if its total discrepancy is larger than $-1+\epsilon$. We say that  a log pair $(Y,D)$ with a morphism $f:Y\to S$  is log smooth over a $S$ if every strata of $(Y,{\rm Supp} (D))$ (including $Y$) is smooth over $S$.
\end{notation}
%Given an $\mathbb R$-divisor $B=\sum b_iB_i$ and a set $I\subset \mathbb R$, we say that $B\in I$ if $b_i\in I$ for all $i$.

\begin{defn}
Let $X$ be a normal variety, $L$ be a $\mathbb{Q}$-Cartier divisor. If we assume $\kappa(L)\ge 0$, then there is a rational map $\phi: X\dasharrow Z$ such that $\dim Z=\kappa (L)$ and $\kappa (L|_{X_z})=0$ where $z$ is a general point of $Z$. This map is defined (up to birational isomorphism) by $|kL|$ for $k>0$ sufficiently divisible.
We say $\phi$ is the {\it Iitaka fbration} of $L$ (see \cite{Lazarsfeld}). \end{defn}
\begin{defn}\label{d-generic}
For an effective $\mathbb{Q}$-divisor $B$ on $X$ and a rational map $X\dasharrow Z$, we say $B$ is {\it big over the generic point $\eta$} of $Z$, if for some resolution $g:X'\to X$ inducing a morphism $X'\to Z$, the restriction of the divisor $B'=g^{-1}_*B+{\rm Ex}(g)$ to the generic fiber $X'_{\eta}$ is big. Alternatively, this is equivalent to saying  that for such $X'$ there is a relatively big divisor $B'$ over $Z$ satisfying $g_*B'=B$. 
\end{defn}

\subsection{Boundedness}
We say that $\mathcal C$ a collection of log pairs is {\it bounded} if there exists a pair $(\mathcal X ,\mathcal B)$ and a projective morphims $f:\mathcal X \to S$ over a finite type variety $S$ such that any $(X,B)\in \mathcal C$ is isomorphic to the 
fiber $(\mathcal X ,\mathcal B)\times _Z z$ of $f$ for some closed point $z\in Z$.
\begin{prop}\label{p-Qfact} Fix $\epsilon >0$.
Assume $\{(X_i,B_i)\}_i$ is a set of $\epsilon$-klt pairs, which corresponds to the fibers over a dense set of points $s_i\in S$ for a bounded family $(\sX,\mathcal{B})$ over a variety $S$. Then there exists a dense open set $U\subset S$ such that $(K_{\sX}+\mathcal{B})|_{\sX\times_S U}$ is $\mathbb{Q}$-Cartier and klt. 
\end{prop}
\begin{proof}After shrinking $S$, we can assume that $(\mathcal{X}, \mathcal{B})$ has a log resolution $\mu:\mathcal{Y}\to \sX $ such that $(\mathcal Y, \mu ^{-1}_* \mathcal B +\mathcal E )$ is log smooth over $S$, where $\mathcal{E}$ be the sum of the exceptional divisors. We run a minimal model program for $K_{\mathcal{Y}}+\mu^{-1}_*\mathcal{B}+(1-\epsilon)\mathcal{E}$ over $\mathcal{X}$ which, by \cite{BCHM10}, terminates with a relative minimal model $g:\sX'\to \sX$. Let  $\nu:\mathcal{Y}\dasharrow \mathcal{X}'$ be the induced birational map. Since $K_{\mathcal{X'}}+\nu_*(\mu^{-1}_*(\mathcal{B})+(1-\epsilon)\mathcal{E})$ is nef over $\mathcal X$, it follows that 
$$(K_{\mathcal{X'}}+\nu_*(\mu^{-1}_*\mathcal{B}+(1-\epsilon)\mathcal{E}))|_{X'_i}$$
is nef over $X_i$. Since $(K_{\mathcal{X'}}+\nu_*(\mu^{-1}_*\mathcal{B}+(1-\epsilon)\mathcal{E}))|_{X'_i}-g_i^*(K_{\sX_i}+\sB_i)$ is an effective divisor whose support equals the exceptional divisor $\nu_* \mathcal{E}|_{X'_i}$, by the negativity lemma, it follows that $\nu_* \mathcal{E}|_{X'_i}=0$ and  $X'_i\to X_i$ is a small birational $(K_{\mathcal{X'}}+\nu_*(\mu^{-1}_*\mathcal{B})|_{X'_i}$-trivial morphism.
Let $\nu:\bar \sX \to \sX$ be the relative log canonical model of $K_{\mathcal{X'}}+\nu_*\mu^{-1}_*\mathcal{B}$.
Since $\nu$ contracts all  $K_{\mathcal{X'}}+\nu_*\mu^{-1}_*\mathcal{B}$.-trivial curves, it follows that $\bar \sX  _i=\sX _i$.
Therefore, after possibly shrinking $S$, we may assume that  $\nu_* \mathcal{E}=0$ and $\nu:\bar \sX \to \sX$ is an isomorphism. In particular $K_\sX+\sB=g_*(K_{\mathcal{X'}}+\nu_*\mu^{-1}_*\mathcal{B})$ is $\mathbb Q$-Cartier.
 Finally, we may assume that $K_\sX+\sB$ is klt as the coefficients of the components of $K_{\mathcal{Y}}-\mu^*(K_{\sX}+\mu ^{-1}_* \mathcal B +\mathcal E)$ are larger than $-1$ since this is true after restricting to fibers $\mathcal Y _i\to \sX_i$.
 %by Inversion of Adjunction (cf. \cite[5.50]{KM98}). 
\end{proof}
\begin{prop} \label{p-de} Fix $\epsilon >0$. Let $(X_i,B_i)$ be a bounded family of $\mathbb Q$-factorial $\epsilon$-klt pairs. Then there exists a family $(\mathcal Y , \mathcal D)\to T$ such that for any $i$ and any set of divisors $E_{i,j}$ with $j\in I_i$ such that $E_{i,j}$ is exceptional of discrepancy $a(E_{i,j},X_i,B_i)\leq 0$, then there exists $t\in T$ and a morphism $\mu _t : \mathcal Y _t \to X_i$ which extracts precisely the divisors $E_{i,j}$ with $j\in I_i$  such that $K_{\mathcal Y _t}+  \mathcal D _t=\mu _t^*(K_{X_i}+B_i)$.
\end{prop}
\begin{proof} Let $(\mathcal X,\mathcal B)\to S$ be a family parametrizing the $(X_i,B_i)$, so that $(X_i,B_i)\cong (\sX,\sB)\times _X {s_i}$ for some $s_i\in S$.  Replacing $S$ by the closure of the $s_i$, we may assume that the points $\{s_i \}$ corresponding to the pairs $(X_i,B_i)$ form a dense  subset in $S$. By Noetherian induction it suffices to  prove the claim over any non-trivial open subset of $S$.
By Proposition \ref{p-Qfact}, we  may assume that $K_{\mathcal{X}}+\mathcal{B}$ is $\mathbb{Q}$-Cartier and klt.   Let 
$\nu:\mathcal X '\to \mathcal X$ be a log resolution and write 
$$K_{\mathcal X'}+\mathcal D'=\nu ^*(K_{\mathcal X}+\mathcal B)+\mathcal E,$$ where $\mathcal D'$ and $\mathcal E$ are effective with no common components. 
By standard arguments, after a dominant base change,  we may assume that $(\mathcal X',\mathcal D')\to S$ is log smooth with geometrical irreducible strata over $S$, terminal and the $(X_i,D_i)$ correspond to a dense subset of points $s_i\in S$ and the induced morphism from the corresponding fiber $\nu _i:(X'_i,D'_i):=(\mathcal X',\mathcal D')\times _S s_i\to (X_i,B_i)$ is a log resolution.

Note that the divisors of discrepancy $a(E,X_i,B_i)\leq 0$ are in one to one correspondence with divisors $\mathcal E$ over $\mathcal X$ such that $a({\mathcal E },\mathcal X,\mathcal B)\leq 0$ and all of these divisors are divisors on $\mathcal X '$. Denote this set of divisors by $\mathcal F$.
By \cite[1.4.3]{BCHM10}, for any subset $\mathcal F '\subset  \mathcal F$ there exists a birational map $\mathcal X '\dasharrow \mathcal X _{\mathcal F '}$ over $\mathcal X$ such that $\mathcal X _{\mathcal F '}\to \mathcal X$ extracts precisely the divisors on $\mathcal F '$. There exists a dense open set $U$ such that for any $t\in T$, $(\mathcal{X}_{\mathcal{F}})_t\to X_t$  extracts precisely the divisor $(\mathcal{E}_i)_t$ for $i\in \mathcal{F}'$. 
Letting $\D'$ be the pushforward of $\D'_{\mathcal{F'}}$ on  $\sX_{\mathcal{F}'}$, defining
$$(\mathcal{Y},\D)|_U\cong (\sX_{\mathcal{F}'},\D'_{\mathcal{F}}) \times_S U, $$  and continuing  by Noetherian induction, we obtain the required family $(\mathcal Y , \mathcal D)\to T$. \end{proof}

\begin{defn}A projective normal variety $X$ is of {\it Fano type}, if there exists an effective big $\mathbb{Q}$-divisor $B$ such that $K_X+B\sim_{\mathbb{Q}}0$ and $(X,B)$ is klt. 
\end{defn}

If $X$ is of Fano type, then for any ample divisor $A$, there exists a sufficiently small rational number $\epsilon>0$, such that $B-\epsilon A\sim_{\mathbb{Q}}G\geq 0$ as $B$ is big. For $0<\delta \ll 1$, let $B'=(1-\delta)B+\delta G$. It follows easily that  $(X,B')$ is klt and since 
$$-(K_X+B')\sim _{\mathbb Q} \delta \epsilon A,$$the pair $(X,B')$ is a log Fano variety. In particular, if $X$ is also $\mathbb Q$-factorial, then by \cite[1.3.2]{BCHM10}, $X$ is a Mori Dream space. Thus the number of SQM (small $\mathbb Q$-factorial modifications) of $X$ is bounded, and each SQM  corresponds to a maximal dimensional chamber of ${\rm Mov}(X)$.

\begin{lem}\label{l-inv}
Let $f:\mathcal{X}\to S$ be a projective family over an affine variety $S$, such that the geometric generic fiber is of Fano type. Then there is a dominant map  $T\to S$ such that for any line bundle $L$ on $\sX$ and any $t\in T$, then for sufficiently divisible $m$, the map
$$f_*( L^{\otimes m})\to H^0(X_t:=\sX\times_S\{t\}, L^{\otimes m}|_{X_t})$$
is surjective. 
\end{lem}
\begin{proof}After replacing $S$ by $T$, we can assume that $S$ is affine and there is a big boundary $\mathcal{B}'$, such that $K_{\sX}+\mathcal{B}'\sim_{\mathbb{Q}}0$ and $(\sX,\mathcal{B}')$ is klt. We may further assume that $\mathcal{B}'\sim _{\mathbb Q} \sA+\sB$ where $\sA$ is ample, $(\sX,\sB)$ is klt and that $(\sX_t,\sB_t)$ is klt for any $t\in S$.  We take a log resolution $\psi:\mathcal{X}'\to \sX$ of $(\sX,\sB)$ and write
$$\psi^*(K_{\sX}+\sB)+\mathcal{E}=K_{\sX'}+\D',$$
where $\mathcal{E}$, $\D'$ are effective and do not have common components. We can assume that $(\sX',\D')$ is terminal.  By shrinking $S$, we may assume that $(\sX',\D')$ is log smooth over $T$ and for any $t\in S$, $\sX'_t\to \sX_t$ is a log resolution of $(\sX_t,\sB_t)$.

Let $L $ be a line bundle on $\sX$, then for any rational number $0<\lambda\ll 1$, the $\mathbb Q$-divisor $\sA_{\lambda}:=\sA+\lambda L$ is ample. By \cite[1.8]{HMX13},  for sufficiently divisible $m$, the morphism
$$H^0(\sX', \mathcal O _{\sX'}(m(K_{\sX'}+\D'+\phi^*\sA_{\lambda})))\to H^0(X'_t,\mathcal O_{X'_t} (m(K_{\sX'}+\D'+\phi^*\sA_{\lambda})|_{X'_t}))$$
is surjective. Since $\psi _* ( K_{\sX'}+\D'+\phi^*\sA_{\lambda})=K_{\sX}+\sB +\sA_{\lambda}\sim _{\mathbb Q} \lambda L$, this implies that
$$H^0(\sX, mL)\to H^0(X_t, mL|_{X_t})$$
is surjective for any $m>0$ sufficiently divisible. 
\end{proof}

\begin{prop}\label{p-mov}
Let $\mathcal{X}/S$ be a projective family, such that $S$ is irreducible and the geometric generic fiber is of Fano type.
Then there exists a finite dominant morphism $T\to S$, such that if  $\mathcal{X}_T= \mathcal{X}\times_ST$ then for any $t\in T$
\begin{enumerate}
\item the morphism $\rho_t:N^1(\sX_ T/T)\to N^1(\sX_t)$ is an isomorphism, 
\item $\rho_t({\rm Mov}(\sX_T/T))={\rm Mov}(\sX_t)$ and there is a one-to-one correspondence between the two Mori chamber decompositions.  
\end{enumerate}
%the Mori chamber decomposition of the mobile cone ${\rm Mov}(\mathcal{X}_T/T)$ is flat over $T$ and fiberwisely yields the Mori chamber decomposition of $\sX_t$ for any $t\in T$.
\end{prop}
\begin{proof}%By Noetherian induction, it suffices to prove this over a non empty open subset of $S$. Therefore we will repeatedly replace $\mathcal X \to S$ by a finite dominant base change i.e. by $\mathcal X \times _S T\to T$ where  $T\to S$ is dominant and finite. 

After possibly replacing $S$ by an appropriate dominant base change, we may assume that $S$ is affine, there is a big  boundary $\mathcal{B}$ over $S$ such that $(\mathcal{X},\mathcal{B})$ is a klt pair, projective over $S$, each fiber   $(\mathcal{X}_s,\mathcal{B}_s)$ is a klt pair and   
$$K_{\sX}+\mathcal{B}\sim_{\mathbb{Q}}0.$$ 
We may assume that a  $\mathbb{Q}$-factorialization  $\nu: \sX^{Q}\to \sX$ is a small morphism which remains small when restricted to any fiber $\sX^{Q}_t\to \sX_t$.  We claim that $N^1(\mathcal X /T)$ and ${\rm Mov }(\mathcal X /T)$ are unchanged under small morphisms (and a similar statement holds for $N^1(\mathcal X _t)$ and ${\rm Mov }(\mathcal X _t)$). Since any divisor (as well as linear equivalence between divisors) is determined by its restriction to a big open subset,  the claim about
${\rm Mov }(\mathcal X /T)$ is clear. Notice that if $D,D'$ are divisors on $\sX$ such that $[D]=[D']$ in  $N^1(\mathcal X /T)$, then $D-D'$ is a $\mathbb Q$-Cartier divisor such that $D-D'\equiv _T0$. But then
it is clear that $\nu ^{-1}_*(D-D')=\nu ^*(D-D') \equiv _T0$. Therefore, we can assume that $\sX$ is $\mathbb{Q}$-factorial. 
By \cite[6.6]{dFH11}, after a base change, we can assume that for any $t\in S$, the morphisms
$$N^1(\sX/S)\to N^1(\sX_t) \mbox{ \  and \ } N_1(\sX/S)\to N_1(\sX_t)$$
are isomorphisms.  
In particular,  $\sX_t$ is $\mathbb{Q}$-factorial. This gives (1).

By Lemma \ref{l-inv}, we can assume that (after a possible base change), 
$$\rho_t: {\rm PSEFF}(X/S)\to  {\rm PSEFF}(\sX_t)$$
is an isomorphism for every $t\in S$ (here $ {\rm PSEFF}(X/S)$ denotes the relative pseudo-effective cone i.e. the closure of the relative big cone). Let $\sX_1, \sX_2,\ldots ,\sX_k$ be the $\mathbb{Q}$-factorial birational models corresponding to the maximal dimensional chambers of the Mori chamber decomposition of ${\rm Mov}(\sX/S)$.  Each $\sX_i$ is also of Fano type, in particular, it is a Mori dream space (see \cite[1.3.2]{BCHM10}). The nef cones ${\rm Nef}(\sX_i/S)$ ($1\le i \le k$) are finitely generated polyhedral cones and  give the Mori chamber decomposition of ${\rm Mov}(\sX/S)$.

By shrinking $S$, we assume that for any $t\in S$, any $1\le i \le k$, $\sX_t\dasharrow (\sX_i)_t$ is a small birational map, which implies 
$$\rho_t({\rm Mov}(\sX/S))\subset {\rm Mov}(\sX_t).$$

Furthermore, by Lemma \ref{l-inv}, we can assume that for any $L$ on $\sX$, for any sufficiently divisible integer $d>0$
$$H^0(\sX,dL)\to H^0(\sX_{t}, dL|_{\sX_{t}}), $$
is surjective for any $t\in S$. 

 If there is a movable divisor class $L_t$ on $X_{t}$, then by our assumption it is the restriction of an effective divisor $L$ on $\sX$. 
 If $L$ is not movable, then the stable base locus ${\bf B}(\sX, L)$ contains a divisor, which implies that ${\bf B}(\sX_t,L|_{\sX_t})$ contains a divisor (since we have assumed that $H^0(\sX,dL)\to H^0(\sX_{t}, dL|_{\sX_{t}})$ is surjective). But this contradicts the assumption that $L_t$ is movable. Therefore, $L$ is movable on $\sX$ and so $\rho_t({\rm Mov}(\sX/S))= {\rm Mov}(\sX_t)$.

%And $\sX_{k+1}$, ..., $\sX_{m}$ the birational models which are obtained by  maximal dimensional chambers of $NE(\sX)$.
%By shrinking $S$, we assume that for any $t\in S$, any $1\le i \le k$, $\sX_t\dasharrow (\sX_i)_t$ is a small birational map and any $k+1\le j\le m$,  $\sX_t\dasharrow (\sX_j)_t$ is not small. This implies 
%$$\rho_t({\rm Mov}(\sX/S))\subset {\rm Mov}(X_t).$$
%Furthermore, by Lemma \ref{l-inv}, we can assume that for any $L$ on $\sX_p$ ($1\le p \le m$), for a sufficiently divisible $d$,
%$$H^0(\sX_p,dL)\to H^0(X_{p,t}, dL|_{X_{p,t}}), $$
%is surjective for any $t\in S$. 

 %If there is a movable divisor class $L_t$ on $X_{t}$, then by our assumption it is the restriction of an effective divisor $L$ on $\sX$.  And there is a model $\sX\dasharrow \sX_p$ such that the push forward of $L$ is nef on $\sX_p$. If $p\le k$, that means $L$ is movable, and so $L|_t$ is in the corresponding chamber. Otherwise, $\sX\dasharrow \sX_p$ is not small, thus the stable base locus ${\bf B}(\sX, L)$ contains a divisor, which implies that ${\bf B}(\sX_t,L|_{X_t})$ contains a divisor. But this contradicts to that $L_t$ is movable. Thus we know $L$ is movable on $\sX$. So $\rho_t({\rm Mov}(\sX/S))= {\rm Mov}(X_t)$.
 
Replacing $S$ by an open subset, we may assume that for any $t\in S $ and $1\le i\le k $ the $\sX_{i,t}$  are mutually not isomorphic, so that the cones ${\rm Nef}(\sX_{i,t})$ give different Mori chambers in ${\rm Mov}(\sX_t)$. This completes the proof of (2).
\end{proof}

\begin{prop}\label{p-sqfm} Let $(X_i,B_i)$ be a bounded family of $\mathbb Q$-factorial klt pairs such that $K_{X_i}+B_i\sim _{\mathbb Q}0$ and $B_i$ is big. Then the set of all small 
%$\mathbb Q$-factorial 
modifications of the $X_i$ is also bounded. \end{prop}
\begin{proof} Let  $(\sX,\mathcal{B})\to S$ be a bounded family which parametrizes all $(X_i,B_i)$.
By Noetherian induction, it suffices to prove the assertion over some non-empty open subset of $S$ such that the points parametrizing $(X_i,B_i)$ form a dense set in $S_i$.

We first only consider small $\mathbb{Q}$-factorial modifications. 
%Let $(\mathcal X,\mathcal D)\to S$ be a family parametrizing the $(X_i,D_i)$.
We may assume that $S$ is affine.% and the $\{(X_i,D_i)\}$ correspond to a dense subsets of points $s_i\in S$, i.e. $(X_i,D_i)\cong (\mathcal X_{s_i},\mathcal D_{s_i})$. 

Fix $a>0$ such that $K_\mathcal X+(1+a)\mathcal B$ is klt. Let $\nu:\mathcal X '\to \mathcal X$ be a log resolution of $(\mathcal X,\mathcal B)$ and write 
$$K_{\mathcal X'}+\mathcal D'=\nu ^*(K_{\mathcal X}+(1+a)\mathcal B)+\mathcal E$$ 
where $\mathcal D'$ and $\mathcal E$ are effective with no common components. 
By standard arguments, we may assume that $(\mathcal X',\mathcal D')\to S$ is log smooth.
By \cite[1.8]{HMX13}, if $m\mathcal D'$ is integral, then 
$$h^0(\mathcal O_{\mathcal X_s}(m(K_{\mathcal X}+(1+a)\mathcal B)|_{\sX_s}))=h^0(\mathcal O_{\mathcal X'_s}(m(K_{\mathcal X'}+\mathcal D')|_{\mathcal{X}'_s}))$$ 
is locally constant for $s\in S$. Therefore $K_{\mathcal X_\eta}+(1+a)\mathcal B_\eta$ is big where $\eta$ is the generic point of any component of $S$. By a similar argument, we also see that $$K_{\mathcal X_\eta}+\mathcal B_\eta\sim _{\mathbb Q}0.$$ Therefore $\mathcal B_\eta$ is big. By shrinking $S$, we can assume that $K_{\sX}+\mathcal{B}\sim_{\mathbb{Q}}0$. 

By Proposition \ref{p-mov}, after possibly shrinking $S$, if $\sX _j$ are the SQMs of $\sX$, then the  $\sX_{j,s_i}$ are the SQMs of the $X_i\cong \sX _{s_i}$.  

We now consider the general case. Let $X'_i$ be any small modification of $X_i$ and $B'_i$ is the pull back of $B_i$. Since $(X_i,B_i)$ is klt and $B_i$ is big, it is easy to see that $(X'_i,B'_i)$ is klt and $B'_i$ is big.  Thus there is a $\mathbb{Q}$-factorialization $\phi:X''_i\to X'_i$ (i.e. a small birational morphism such that $X''_i$ is $\mathbb{Q}$-factorial \cite{BCHM10}). Pick $\sX_p$ a SQM of $\sX$, such that 
$$\sX_{p,s_i}\cong X''_i.$$ 
The pull back of ${\rm Nef}(X_i')$ corresponds to a face $\Sigma$ of ${\rm Nef}(X''_{i})$ which in turn corresponds to a face $\Sigma$ of ${\rm Nef}(\sX_p)$. Pick a general $\mathbb Q$-divisor $L$ in the open chamber $\Sigma^0$. By \eqref{l-inv}, $L$ is big and so it induces a morphism $\sX_p \to \sX'_p$, where $X_i'\cong \sX'_{p,s_i}$. Since ${\rm Nef}( \sX_p)$ consists of finitely many polyhedral chambers, we conclude that there are only finitely many possible birational models $\mathcal X'_p$ with a morphism $\mathcal{X}_p\to \mathcal{X}'_p$ and hence the small modifications $X_i$ are also bounded.
\end{proof}

\section{The main result}
Fix $I\subset [0,1]$ a DCC set and $n\in \mathbb N$.
Let $(X,B)$ be a klt pair such that the restriction of $B$ to the generic fiber of the Iitaka fibration is big, $\dim X =n$ and $B\in I$. 
The goal is to find an effective integer $m=m(n,I)$ such that $|m(K_X+B)|$ induces the Iitaka fibration. If $\kappa (K_X+B)=0$ this simply means $h^0(\mathcal O _X(m(K_X+B)))\ne 0$.
Recall the following result from \cite[1.7]{HMX14}.
\begin{prop}\label{t-bd1} 
Fix $I\subset [0,1]\cap \mathbb Q$ a DCC set and $n\in \mathbb N$.  If $(X,B)$ is an $n$-dimensional klt pair with $K_X+B\equiv 0$, $-K_X$ is ample and the coefficients of $B$ are contained in $I$, then $(X,B)$ is in a bounded family. 
%In particular, Then there is a finite set $I_0\subset I$ such thatcoefficients of $B$ are contained in $I_0$.
\end{prop}
\begin{proof} {\cite[1.7]{HMX14}} proves the statement assuming that there is an $\epsilon >0$ such that the total log discrepancy of $(X,B)$ is grater than $\epsilon$. Thus it suffices to prove that such an $\epsilon >0$ exists.

Suppose that $(X_i,B_i)$ is a sequence of pairs as above with $\lim \epsilon _i=0$ where $\epsilon _i>0$ is the total log discrepancy of $(X_i,B_i)$. Passing to a subsequence, we may assume that $\epsilon _i$ is decreasing. Let $(X_i',B'_i)\to (X_i,B_i)$ be the pairs obtained by extracting a divisor of log discrepancy $\epsilon _i$, then the coefficients of $B'_i$ belong to the DCC set $J=I\cup \{1-\epsilon _i|i\in \mathbb N \}$. Hence the pairs $(X_i',B'_i)$ violate \cite[1.5]{HMX14}, a contradiction. Thus there is an $\epsilon =\epsilon (n,I)>0$ such that $\epsilon _i>\epsilon$ for all $i\in \mathbb N$ and we are done by \cite[1.7]{HMX14}.
\end{proof}

\begin{proof}[Proof of Theorem \ref{t-k=0}]  By \cite{BCHM10}, running the $(K_X+B)$-mmp with scaling, we may assume that $(X,B)$ is a good minimal model, i.e. that $K_X+B\sim _{\mathbb Q }0$. 

Let $X\dasharrow X'$ be a minimal model for $K_X+(1+a)B$ and $h:X'\to Z$ the corresponding log canonical model where $0<a\ll 1$ is chosen so that $(X,(1+a)B)$ is klt. %Note that $K_{X'}+B'\sim _{\mathbb Q }0$ and so by \cite[Theorem B]{HMX14}, ${\rm vol}(X',B')$ is bounded (from above). 
By boundedness of log canonical models \cite[1.1.5]{BCHM10}, we may assume that $Z$ is independent of $a$. In particular $h_* B'$ is $\mathbb Q$-Cartier. It follows that $-K_{Z}\sim _{\mathbb Q} h_* B'$ is ample and hence $(Z, h_* B')$ is bounded by \eqref{t-bd1}.

Consider a common resolution $p:Y\to X$ and $q:Y\to Z$, since
$$p^*(K_X+B)=q^*(K_Z+h_*B'),$$
we know that for any exceptional divisor $E$ of $X/Z$, the discrepancy
$$a(E,Z,h_*B')\le 0.$$ Since $(Z,h_*(B)')$ is $\epsilon$-lt (see the proof of  \eqref{t-bd1}),  then by \eqref{p-de} we can construct a model  
 % Since $K_{X'}+(1+a)B'=h^*(K_Z+(1+a)h_*B')$, it follows that $a_E(Z,(1+a)h_*B')={\rm mult }_E(K_X-h^*(K_Z+(1+a)h_*B'))\geq 0$ and taking the limit as $a\to 0$ we have $a_E(Z,h_*B')\geq 0$. By  \eqref{p-de}  it follows that $(X',B')$ belongs to a bounded family. The map $X\dasharrow X'$ is $B$-negative and hence only contracts divisors contained in the support of $B$. These divisors also have log discrepancy $a_E(X',B')\geq 0$ (since $X\dasharrow X'$ is $K_X+B$-trivial). Thus there is a sequence of divisorial extractions 
 $g:X''\to Z$ such that $\psi: X\dasharrow X''$ is an isomorphism in codimension $1$ and $(X'',B'')$ is bounded where $B''=\psi _* B$ so that $K_{X''}+B''=g^*(K_{Z}+h_*B')$. Note that as $B$ is big, so is $B''$ and  $(X'',B'')$ is klt, $\mathbb Q$-factorial with $K_{X''}+B''\sim _{\mathbb Q}0$. By \eqref{p-sqfm} it follows that $(X,B)$ belongs to a bounded family.
\end{proof}

\begin{proof}[Proof of \eqref{c}]By \cite{Kollar93}, we know that there exists a uniform $N=N(m,n)$ such that $|-NK_X|$ is base point free. Let $D$ be a general divisor in $|-NK_X|$, then 
the pair $(X,\frac{1}{N}D)$ satisfies the assumption of Theorem \ref{t-k=0} and the assertion follows.
% Let $X^{s}$ be the anti-canonical model of $X$. Then $X^s$ is an $n$-dimensional $\mathbb{Q}$-Fano variety with Cartier index $m$. Then \cite[1.8]{HMX14} implies all such $X^s$ are contained in a bounded family. Since the divisors contracted by $X\to X^s$ are all of discrepancy 0. The results follow from Proposition \ref{p-de} and \ref{p-sqfm}. 
\end{proof}

\begin{proof}[Proof of \eqref{t-bd}]
 By \cite[4.4]{Lai11}, we may replace $(X,B)$ by its good minimal model and hence we may assume that $K_X+B$ is semiample. Let $Z={\rm Proj} (R(K_X+B))$  and $f:X\to Z$  the induced morphism which coincides with the Iitaka fibration.
We can write $K_X+B\sim _{\mathbb Q}f^*(K_Z+B_Z+M_Z)$ where $B_Z$ is the boundary part and $M_Z$ is the moduli part. 

Let $\nu :Z'\to Z$ be a resolution of $Z$ and $X''$ be a resolution of the main component 
of $X\times _ZZ'$, $\mu :X''\to X$ the induced morphism. Write $$K_{X''}+B_{X''}=\mu ^*(K_X+B)+\sum a_i E_i$$
such that $E_i$ is exceptional, and the coefficients of $B_{X''}$ are contained in the DCC set $I\cup \{\frac{n-1}{n}|n\in \mathbb{N}\}$ and $a_i> 0$. 
$B''$ is big over $Z$ and so we can run a relative MMP of $(X'',B'')$ over $Z'$. Let $(X',B')\to Z$ be the corresponding minimal model over $Z$.  

We may write $K_{X'}+B_{X'}\sim _{\mathbb Q}{f'}^*(K_{Z'}+B_{Z'}+M_{Z'})$ where $B_{Z'}$ is the boundary part and $M_{Z'}$ is the moduli part. We may assume that $M_{Z'}$ is nef (cf. \cite[Theorem 2]{Kawamata98}). 
\begin{claim} There exists an integer $r>0$ and a DCC set $\Lambda$ such that
\begin{enumerate}
\item $(Z',B_{Z'})$ is klt,
\item the coefficients of $ B_{Z'}$ belong to $\Lambda$,
\item $rM_{Z'}$ is a nef Cartier divisor, and 
\item $K_{Z'}+B_{Z'}+M_{Z'}$ is big.
\end{enumerate}\end{claim}
%Since $f$ is birational to the Iitaka fibration of $(X,B_{X})$, $K_Z+B_Z+M_Z$ is big. Since  $K_{Z'}+B_{Z'}^{\geq 0}+M_{Z'}\sim _{\mathbb Q} \nu ^*(K_Z+B_Z+M_Z)+B_{Z'}^{< 0}$ and $B_{Z'}^{< 0}$ is effective and $\nu$-exceptional, it follows that (4) holds.
Since $f'$ is birational to the Iitaka fibration of $(X,B_{X})$ and 
$$K_{X'}+B_{X'}\sim _{\mathbb Q}{f'}^*(K_{Z'}+B_{Z'}+M_{Z'}),$$ it follows that (4) holds.

To check that (3) holds, we follow \cite[\S 3]{Todorov10}. Let $F$ be the general fiber of $f$ and write 
$(K_X+B)|_F=(K_F+B_F)$, then $K_F+B_F\sim _{\mathbb Q}0 $ and $B_F$ is big so that by \eqref{t-k=0}, $(F,B_F)$ belongs to a bounded family. In particular, there exists a $b>0$ such that $b(K_F+B_F)\sim 0$. We let $\eta: F''\to F$ be a log resolution of $(F,B_F)$. We write 
$$K_{F''}+B_{F''}=K_{F''}+B_{F''}=\eta ^* (K_F+B_F),$$ then
$$ D_{F''}:=b\{B_{F''}\}\sim_{\mathbb{Q}}b (-K_{F''}-\lfloor B_{F''} \rfloor). $$
 In particular,  $D_{F''}$ is an effective divisor.

We denote by  $\phi: F'''\to  F''$ and  the  corresponding cyclic covering and $\psi:F'\to F'''$ a resolution. Since $F$  is in a bounded family, it is easy to see that we can assume that $F'$ is in a bounded family. In particular, if we denote the dimension of $F$ to be $m$, then its middle Betti number $B_m(F')$ is bounded by a constant $l$. 
\begin{lem}Let $\varphi$ be the Euler function, then there exists a positive integer $r$ such that $rM_{Z'}$ is Cartier and
$\varphi(r)\le l$.
\end{lem}
\begin{proof}When $F$ is smooth and $B_F=0$, this is \cite[Theorem 3.1]{FM00}. The general case is well known and similar to \cite{FM00}.  For the reader's convenience, we include a short sketch.  

Following \cite[3.6]{FM00}, fix any divisor $P$ on $Z'$, let $(Q\in W)\to (P\in Z')$ be a base change, such that $X'\times_{Z'} W$ admits a semistable model over a neighborhood of the generic point of $Q$. 
Then $rM_{Z'}$ is integral at $P$ if the local monodromy action  on $H^0(F', K_{F'})$ (where $F'$ is considered as a fiber over $W$) factors through 
$$\mu_r: \mathbb{Z}/r\mathbb{Z}\to \mathbb{C}^*.$$ 

On the other hand, there is a $\mu_r$-equivariant injection
$$H^0(F', K_{F'}) \subset  H^m(F',\mathbb{Z})\otimes \mathbb{C}. $$
The eigenvalues of the corresponding action on $H^m(F',\mathbb{Z})\otimes \mathbb{C}$ are always algebraic integers, thus we conclude $\varphi(r)\le \dim _{\mathbb C}(H^m(F',\mathbb{Z})\otimes \mathbb{C})=l $.
\end{proof}

To check that (2) holds, recall that for a codimension 1 point $P$ on $Z$, its coefficient in $B_{Z'}$ is equal to
$$1-{\rm lct}_{f'^{-1}(P)}(X',Z',f'^{*}(P)), $$
where the log canonical threshold is computed locally around $f'^{-1}(P)$  (cf. \cite{Kawamata98}, \cite[3.1.4]{Ambro99}). Therefore, by \cite[Theorem 1.1]{HMX14}, it is contained in a DCC set. 
%It is easy to see that $\sum _{i=0}^bh^0(\omega _{F'}^{1-i}(-\lfloor iD_F \rfloor ))$ is bounded and so (3) holds.
%To check that (2) holds, consider $v:X''\to X'$ a log resolution such that for any codimension $1$ point $Q$ on $Z'$, the pair $(X'', B_{X''}^{> 0} +f^*Q )$ is simple normal crossings over an open subset of $Z'$ which intersects $Q$. Then the coefficients of $B_{Z'}$ are of the form $${\rm max} _j \frac {b_j+w_j-1}{w_j}$$ where ${f''}^*Q=\sum w_jP_j$ and $b_j={\rm mult }_{P_j}(B_{X''}^{>0})$ (cf. \cite[3.1.4]{Ambro99}).
%Since the pairs $b(K_F+B_F)\sim 0$, it follows that $b(K_{F''}+B_{F''})$ is Cartier. In particular the coefficients of $b(K_{X''}+B_{X''}^{>0})|_{F''}$ are integers and since $F''$ is a general fiber, the coefficients of $b(K_{X''}+B_{X''}^{>0})$ are also integers. (2) follows and the claim is proven.

By \cite[1.2]{BZ14}, it follows that there is an integer $m$ depending only on $\Lambda $, $r$ and $d=\dim Z$ such that $|m(K_{Z'}+B_{Z'}+M_{Z'})|$ defines a birational map. Since $K_{X'}+B_{X'}\sim _{\mathbb Q} {f'}^*(K_{Z'}+B_{Z'}+M_{Z'})$ it follows that $|m(K_{X'}+B_{X'})|$ defines the Iitaka fibration. This concludes the proof of \eqref{t-k=0}.
\end{proof}

%%%%%%%%%%%%%%%%%%%%%%%%%%%%%%%%%%%%%%%%%%%

%%%%%%%%%%%%%%%%%%%%%%%%%%%%%%%%%%%%%%%%%%%

\bigskip


\begin{bibdiv}
\begin{biblist}%[\normalsize]

% \bib{AK00}{article}{
%     AUTHOR = {Abramovich, Dan},
%   AUTHOR = {Karu, Kalle},
%      TITLE = {Weak semistable reduction in characteristic 0},
%    JOURNAL = {Invent. Math.},
%   FJOURNAL = {Inventiones Mathematicae},
%     VOLUME = {139},
%       YEAR = {2000},
%     NUMBER = {2},
%      PAGES = {241--273},
%  }


%\bib{AKMW02}{article}{
  % author={Abramovich, Dan},
  % author={Karu, Kalle},
  % author={Matsuki, Kenji},
  % author={W{\l}odarczyk, Jaros{\l}aw},
   %title={Torification and factorization of birational maps},
   %journal={J. Amer. Math. Soc.},
   %volume={15},
   %date={2002},
   %number={3},
   %pages={531--572 (electronic)},
%   issn={0894-0347},
%   review={\MR{1896232 (2003c:14016)}},
%   doi={10.1090/S0894-0347-02-00396-X},
%}

\bib{Alexeev94}{article}{
    AUTHOR = {Alexeev, Valery},
     TITLE = {Boundedness and {$K^2$} for log surfaces},
   JOURNAL = {Internat. J. Math.},
  FJOURNAL = {International Journal of Mathematics},
    VOLUME = {5},
      YEAR = {1994},
    NUMBER = {6},
     PAGES = {779--810},
 }

%\bib{Ambro03}{article}{
   % AUTHOR = {Ambro, Florin},
    % TITLE = {Quasi-log varieties},
   %JOURNAL = {Tr. Mat. Inst. Steklova},
  %FJOURNAL = {Trudy Matematicheskogo Instituta Imeni V. A. Steklova.
     %         Rossi\u\i skaya Akademiya Nauk},
    %VOLUME = {240},
     % YEAR = {2003},
    %NUMBER = {Biratsion. Geom. Linein. Sist. Konechno Porozhdennye Algebry},
     %PAGES = {220--239},
      %ISSN = {0371-9685},
   %MRCLASS = {14E30 (14J10)},
 % MRNUMBER = {MR1993751 (2004f:14027)},
%MRREVIEWER = {Tomasz Szemberg},
%}

%\bib{Ambro04}{article}{
  % author={Ambro, Florin},
   %title={Shokurov's boundary property},
   %journal={J. Differential Geom.},
   %volume={67},
   %date={2004},
   %number={2},
   %pages={229--255},
%   issn={0022-040X},
%   review={\MR{2153078 (2006d:14033)}},
%}
\bib{Ambro99}{article}{
  author={Ambro, Florin},
   title={The Adjunction Conjecture and its applications},
   journal={arXiv:9903060},
   %volume={67},
   date={1999},
   %number={2},
   %pages={229--255},
%   issn={0022-040X},
%   review={\MR{2153078 (2006d:14033)}},
}

\bib{BCHM10}{article}{
   author={Birkar, Caucher},
   author={Cascini, Paolo},
  author={Hacon, Christopher D.},
  author={McKernan, James},
   title={Existence of minimal models for varieties of log general type},
   journal={J. Amer. Math. Soc.},
   volume={23},
   date={2010},
   number={2},
   pages={405--468},
%   issn={0894-0347},
%   review={\MR{2601039 (2011f:14023)}},
%   doi={10.1090/S0894-0347-09-00649-3},
}

\bib{BL14}{article}{
   author={Borisov, Lev},
   author={Li, Zhani},
   title={On complete intersection with trivial canonical class},
   journal={ArXiv:1404.7490},
year={2014}, 
%   issn={0894-0347},
%   review={\MR{2601039 (2011f:14023)}},
%   doi={10.1090/S0894-0347-09-00649-3},
}

\bib{BZ14}{article}{
   author={Birkar, Caucher},
   author={Zhang, De-Qi},
   title={Effectivity of Iitaka fibrations and pluricanonical systems of polarized pairs},
   journal={ArXiv:1410.0938},
year={2014}, 
%   issn={0894-0347},
%   review={\MR{2601039 (2011f:14023)}},
%   doi={10.1090/S0894-0347-09-00649-3},
}

\bib{CC}{article}{
    AUTHOR = {Chen, Jungkai},
    AUTHOR={Chen, Meng},
     TITLE = {Explicit birational geometry of threefolds of general type,
              {I}},
   JOURNAL = {Ann. Sci. \'Ec. Norm. Sup\'er. (4)},
  FJOURNAL = {Annales Scientifiques de l'\'Ecole Normale Sup\'erieure.
              Quatri\`eme S\'erie},
    VOLUME = {43},
      YEAR = {2010},
    NUMBER = {3},
     PAGES = {365--394},
 }
	\bib{CL12}{article}{
    AUTHOR = {Cascini, Paolo}
    AUTHOR ={ Lazi{\'c}, Vladimir},
     TITLE = {New outlook on the minimal model program, {I}},
   JOURNAL = {Duke Math. J.},
  FJOURNAL = {Duke Mathematical Journal},
    VOLUME = {161},
      YEAR = {2012},
    NUMBER = {12},
     PAGES = {2415--2467},
    }

%\bib{Corti07}{collection}{
  % title={Flips for 3-folds and 4-folds},
  % series={Oxford Lecture Series in Mathematics and its Applications},
  % volume={35},
  % editor={Corti, Alessio},
  % publisher={Oxford University Press},
  % place={Oxford},
  % date={2007},
  % pages={x+189},
%   isbn={978-0-19-857061-5},
%   review={\MR{2352762 (2008j:14031)}},
%   doi={10.1093/acprof:oso/9780198570615.001.0001},
%}

\bib{dFH11}{article}{
    AUTHOR = {de Fernex, Tommaso},
    AUTHOR={ Hacon, Christopher D.},
     TITLE = {Deformations of canonical pairs and {F}ano varieties},
   JOURNAL = {J. Reine Angew. Math.},
    VOLUME = {651},
      YEAR = {2011},
     PAGES = {97--126},
 }

\bib{FM00}{article}{
   AUTHOR = {Fujino, Osamu},
   author={Mori, Shigefumi}
     TITLE = {A canonical bundle formula},
 VOLUME = {56},
     PAGES = {167-188},
     YEAR = {2000},
   journal={J. Differential Geometry}
   %MRCLASS = {14E30},
  %MRNUMBER = {MR2359341},
}


\bib{HM06}{article}{
   author={Hacon, Christopher D.},
   author={McKernan, James},
   title={Boundedness of pluricanonical maps of varieties of general type},
   journal={Invent. Math.},
   volume={166},
   date={2006},
   number={1},
   pages={1--25},
%   issn={0020-9910},
%   review={\MR{2242631 (2007e:14022)}},
%   doi={10.1007/s00222-006-0504-1},
}


%\bib{HX09}{article}{
  % author={Hogadi, Amit},
  % author={Xu, Chenyang},
  % title={Degenerations of rationally connected varieties},
  % journal={Trans. Amer. Math. Soc.},
  % volume={361},
  % date={2009},
  % number={7},
  % pages={3931--3949},
%   issn={0002-9947},
%   review={\MR{2491906 (2010i:14091)}},
%   doi={10.1090/S0002-9947-09-04715-1},
%}

\bib{HMX13}{article}{
   author={Hacon, Christopher D.},
     author={McKernan, James},
   author={Xu, Chenyang},
TITLE = {On the birational automorphisms of varieties of general type},
   JOURNAL = {Ann. of Math. (2)},
  FJOURNAL = {Annals of Mathematics. Second Series},
    VOLUME = {177},
      YEAR = {2013},
    NUMBER = {3},
     PAGES = {1077--1111},
}
\bib{HMX14}{article}{
   author={Hacon, Christopher D.},
     author={McKernan, James},
   author={Xu, Chenyang},
 TITLE = {A{CC} for log canonical thresholds},
   JOURNAL = {Ann. of Math. (2)},
  FJOURNAL = {Annals of Mathematics. Second Series},
    VOLUME = {180},
      YEAR = {2014},
    NUMBER = {2},
     PAGES = {523--571},
}
%\bib{HMX15}{article}{
 %  author={Hacon, Christopher D.},
  %   author={McKernan, James},
  % author={Xu, Chenyang},
 %  title={Boundedness of moduli of varieties of general type},
 %  journal={ },
%volum={},
%pages={},
  % date={2015}
%}

\bib{Jiang13}{article}{
    AUTHOR = {Jiang, Xiaodong},
     TITLE = {On the pluricanonical maps of varieties of intermediate Kodaira dimension}
   JOURNAL = {Math. Ann.},
  FJOURNAL = {Mathematische Annalen},
    VOLUME = {356},
      YEAR = {2013},
    NUMBER = {3},
     PAGES = {979--1004},
 }
 
\bib{Kawamata86}{article}{
    AUTHOR = {Kawamata, Yujiro},
     TITLE = {On the plurigenera of minimal algebraic {$3$}-folds with $K_X\equiv 0$}
   JOURNAL = {Math. Ann.},
  FJOURNAL = {Mathematische Annalen},
    VOLUME = {275},
      YEAR = {1986},
    NUMBER = {4},
     PAGES = {539--546},
 }
 
\bib{Kawamata98}{article}{
    AUTHOR = {Kawamata, Yujiro},
     TITLE = {Subadjunction of log canonical divisors. {II}},
   JOURNAL = {Amer. J. Math.},
  FJOURNAL = {American Journal of Mathematics},
    VOLUME = {120},
      YEAR = {1998},
    NUMBER = {5},
     PAGES = {893--899},
}
	
% \bib{Kawamata11}{article}{
%     AUTHOR = {Kawamata, Yujiro},
%      TITLE = {Semipositivity theorem for reducible algebraic fiber spaces},
%    JOURNAL = {Pure Appl. Math. Q.},
%   FJOURNAL = {Pure and Applied Mathematics Quarterly},
%     VOLUME = {7},
%       YEAR = {2011},
%     NUMBER = {4, Special Issue: In memory of Eckart Viehweg},
%      PAGES = {1427--1447},
%  }


\bib{Kollar93}{article}{
    AUTHOR = {Koll{\'a}r, J{\'a}nos},
     TITLE = {Effective base point freeness},
   JOURNAL = {Math. Ann.},
  FJOURNAL = {Mathematische Annalen},
    VOLUME = {296},
      YEAR = {1993},
    NUMBER = {4},
     PAGES = {595--605},
      ISSN = {0025-5831},
}
	

\bib{Kollar94}{incollection} {
    AUTHOR = {Koll{\'a}r, J{\'a}nos},
     TITLE = {Log surfaces of general type; some conjectures},
 BOOKTITLE = {Classification of algebraic varieties ({L}'{A}quila, 1992)},
    SERIES = {Contemp. Math.},
    VOLUME = {162},
     PAGES = {261--275},
 PUBLISHER = {Amer. Math. Soc., Providence, RI},
      YEAR = {1994},
 }

%\bib{KK10}{article}{
  % author={Koll{\'a}r, J{\'a}nos},
  % author={Kov{\'a}cs, S{\'a}ndor J.},
  % title={Log canonical singularities are Du Bois},
   %journal={J. Amer. Math. Soc.},
   %volume={23},
   %date={2010},
   %number={3},
   %pages={791--813},
%   issn={0894-0347},
%   review={\MR{2629988 (2011m:14061)}},
%   doi={10.1090/S0894-0347-10-00663-6},
%}



%\bib{KKMSD73}{book}{
 % author={Kempf, G.},
 % author={Knudsen, F.F.},
 % author={Mumford, D.},
 % author={Saint-Donat, B.},
 % title={Toroidal embeddings. I},
 % series={Lecture Notes in Mathematics, Vol. 339},
%  publisher={Springer-Verlag},
%  place={Berlin},
 % date={1973},
 % pages={viii+209},
  %review={\MR{0335518 (49 \#299)}},
%}



\bib{KM98}{book}{
   author={Koll{\'a}r, J{\'a}nos},
   author={Mori, Shigefumi},
   title={Birational geometry of algebraic varieties},
   series={Cambridge Tracts in Mathematics},
   volume={134},
   note={With the collaboration of C. H. Clemens and A. Corti;
   Translated from the 1998 Japanese original},
   publisher={Cambridge University Press},
   place={Cambridge},
   date={1998},
   pages={viii+254},
   
   
%   isbn={0-521-63277-3},
%   review={\MR{1658959 (2000b:14018)}},
%   doi={10.1017/CBO9780511662560},
}


%\bib{K-etal92}{book}{
   % AUTHOR = {Koll{\'a}r, J{\'a}nos},
   % TITLE =  {Flips and abundance for algebraic threefolds},
     %  NOTE = {Papers from the Second Summer Seminar on Algebraic Geometry
        %      held at the University of Utah, Salt Lake City, Utah, August
          %    1991,
            %  Ast\'erisque No. 211 (1992)},
 %PUBLISHER = {Soci\'et\'e Math\'ematique de France},
    %   YEAR = {1992},
    % PAGES = {115--126},
    %  ISSN = {0303-1179},
  % MRCLASS = {14E30 (14E35 14M10)},
 % MRNUMBER = {94f:14013},
%MRREVIEWER = {Mark Gross},
%}


\bib{Lai11}{article} {
    AUTHOR = {Lai, Ching-Jui},
     TITLE = {Varieties fibered by good minimal models},
   JOURNAL = {Math. Ann.},
  FJOURNAL = {Mathematische Annalen},
    VOLUME = {350},
      YEAR = {2011},
    NUMBER = {3},
     PAGES = {533--547},
  }
  

\bib{Lazarsfeld}{book}{
    AUTHOR = {Lazarsfeld, Robert},
     TITLE = {Positivity in algebraic geometry. {I}},
    SERIES = {Ergebnisse der Mathematik und ihrer Grenzgebiete. 3. Folge. A
              Series of Modern Surveys in Mathematics [Results in
              Mathematics and Related Areas. 3rd Series. A Series of Modern
              Surveys in Mathematics]},
    VOLUME = {48},
      NOTE = {Classical setting: line bundles and linear series},
 PUBLISHER = {Springer-Verlag, Berlin},
      YEAR = {2004},
     PAGES = {xviii+387},
      ISBN = {3-540-22533-1},
   MRCLASS = {14-02 (14C20)},
  MRNUMBER = {2095471 (2005k:14001a)},
MRREVIEWER = {Mihnea Popa},
       DOI = {10.1007/978-3-642-18808-4},
       URL = {http://dx.doi.org/10.1007/978-3-642-18808-4},
}
	

\bib{PS06}{article}{
   author={Prokhorov, Yu. G.  },
     author={Shokurov, V. V.},
   title={Toward the Second Main Theorem on Complements: From Local to Global},
   journal={MPI / Max-Planck-Institut f\"ur Mathematik, Bonn },
volum={},
pages={},
   date={2006}
}

\bib{Takayama06}{article}{
   author={Takayama, Shigeharu},
   title={Pluricanonical systems on algebraic varieties of general type.},
   journal={Invent. Math.},
volum={165},
pages={551-587},
   date={206}
}


\bib{Todorov10}{article}{
    AUTHOR = {Todorov, Gueorgui Tomov},
     TITLE = {Effective log {I}itaka fibrations for surfaces and threefolds},
   JOURNAL = {Manuscripta Math.},
  FJOURNAL = {Manuscripta Mathematica},
    VOLUME = {133},
      YEAR = {2010},
    NUMBER = {1-2},
     PAGES = {183--195},
}

\bib{Tsuji06}{article}{
    AUTHOR = {Tsuji, Hajime},
     TITLE = {Pluricanonical systems of projective varieties of general   type. {I}},
   JOURNAL = {Osaka J. Math.},
  FJOURNAL = {Osaka Journal of Mathematics},
    VOLUME = {43},
      YEAR = {2006},
    NUMBER = {4},
     PAGES = {967--995},
}
	

\bib{TX09}{article}{
    AUTHOR = {Todorov, Gueorgui}
    AUTHOR={Xu, Chenyang},
     TITLE = {Effectiveness of the log {I}itaka fibration for 3-folds and
              4-folds},
   JOURNAL = {Algebra Number Theory},
  FJOURNAL = {Algebra \& Number Theory},
    VOLUME = {3},
      YEAR = {2009},
    NUMBER = {6},
     PAGES = {697--710}
 }
	


\bib{VZ00}{article}{
    AUTHOR = {Viehweg, Eckart}
    AUTHOR={Zhang, De-Qi},
     TITLE = {Effective {I}itaka fibrations},
   JOURNAL = {J. Algebraic Geom.},
  FJOURNAL = {Journal of Algebraic Geometry},
    VOLUME = {18},
      YEAR = {2009},
    NUMBER = {4},
     PAGES = {711--730},
  }







\end{biblist}
\end{bibdiv}
\end{document}